\documentclass[psamsfonts,12pt]{amsart}
\usepackage{mathpazo}
\usepackage{overpic}
\usepackage{amsmath,amsthm,amssymb,amscd,amsfonts,amsbsy}
\usepackage{color}
\usepackage{hyperref,url}
\usepackage{float}
\usepackage{hyperref}
\usepackage{enumerate}
\usepackage{enumitem}   
\usepackage{mathrsfs}  
\usepackage{mathrsfs}  
\usepackage{xfrac}
\usepackage{amsthm}
\usepackage{amsmath}
\usepackage{amsfonts}
\usepackage{amssymb}
\usepackage{amsthm}
\usepackage{amsmath}
\usepackage{amsfonts}
\usepackage{amssymb}
\usepackage{mathtools}
\usepackage{mathrsfs}
\usepackage{comment}
\usepackage{soul}

\usepackage{amsmath, amssymb, graphics, setspace}
\usepackage[utf8]{inputenc}
\usepackage{listings,xcolor}

\newcommand{\mathsym}[1]{{}}
\newcommand{\unicode}[1]{{}}

\usepackage[normalem]{ulem}

\newcommand{\R}{\ensuremath{\mathbb{R}}}
\newcommand{\CC}{\mathcal{C}}

\newcommand{\CP}{\ensuremath{\mathcal{P}}}
\newcommand{\ov}{\overline}
\newcommand{\N}{\mathbb{N}}
\newcommand{\s}{\mathbb{S}}

\newcommand{\Tn}{\mathbb{T}_n}
\newcommand{\inte}{\mathrm{int}}

\newcommand{\e}{\varepsilon}

\newcommand{\sgn}{\mathrm{sign}}

\usepackage[toc,page]{appendix}
\usepackage{tikz-cd}

\newtheorem {theorem} {Theorem}

\newtheorem {proposition} [theorem]{Proposition}
\newtheorem {corollary}{Corollary}
\newtheorem {lemma}  [theorem]{Lemma}

\newtheorem {remark}{Remark}

\newtheorem {mtheorem} {Theorem}

\newcommand\blfootnote[1]{%
  \begingroup
  \renewcommand\thefootnote{}\footnote{#1}%
  \addtocounter{footnote}{-1}%
  \endgroup
}

\def\R{\mathbb R}
\def \d {\mathrm{d}}

\title[On the boundedness of solutions and existence of invariant tori]
{Invariant tori and boundedness of solutions\\ of non-smooth oscillators with \\ Lebesgue integrable forcing term}
\author[D. D. Novaes and L. V. M. F. Silva]
{Douglas D. Novaes$^*$ and Luan V. M. F. Silva}

\address{Departamento de Matem\'{a}tica - Instituto de Matem\'{a}tica, Estat\'{i}stica e Computa\c{c}\~{a}o Cient\'{i}fica (IMECC) - Universidade
Estadual de Campinas (UNICAMP), \ Rua S\'{e}rgio Buarque de Holanda, 651, Cidade Universit\'{a}ria Zeferino Vaz, 13083-859, Campinas, SP,
Brazil}
\email{ddnovaes@unicamp.br}
\email{luanmattos@ime.unicamp.br}

\allowdisplaybreaks
\begin{document}


\maketitle

\begin{abstract}
Since Littlewood works in the 1960's, the boundedness of solutions of Duffing-type equations $\ddot{x}+g(x)=p(t)$ has been extensively investigated. More recently, some researches have focused on the family of non-smooth forced oscillators $ \ddot{x}+\sgn(x)=p(t)$, mainly because it represents a simple limit scenario of Duffing-type equations for when $g$ is bounded. Here, we provide a simple proof for the boundedness of solutions of the non-smooth forced oscillator in the case that the forcing term $p(t)$ is a $T$-periodic Lebesgue integrable function with vanishing average. We reach this result by constructing a sequence of invariant tori whose union of their interiors covers all the $(t,x,\dot x)$-space, $(t,x,\dot x)\in \s^1\times\R^2$.
\end{abstract}

\noindent {\footnotesize {\it Key words and phrases.} Invariant tori, boundedness of solutions, non-smooth oscillators,\\ Carathéodory differential systems}
\blfootnote{$^*$Corresponding author: Douglas D. Novaes, ddnovaes@unicamp.br} 

\section{Introduction}

In this paper, we are interested in the following family of non-smooth forced oscillators:
\begin{equation}\label{e1}
	\ddot{x}+\sgn(x)= p(t).
\end{equation}
This family has been extensively studied in the research literature. For instance, by assuming that $ p(t) $ is a periodic continuous function, the existence of periodic solutions has been analyzed by Jacquemard and Teixeira in \cite{Jacquemard2012}, while Burra and  Zanolin in \cite{Burra2020} and Sun and Sprott  in \cite{Sprott2010} determined the existence of chaotic behavior of equation \eqref{e1} for particular examples of periodic piecewise constant functions $p(t)$. Dropping the periodicity condition on $p(t)$, the boundedness of solutions of \eqref{e1} has been investigated by Enguiça and Ortega in \cite{Ortega2019}, which showed that it has infinitely many bounded solutions provided that the limit
\[
\lim_{T\to+\infty}\frac{1}{T}\int_{t}^{t+T}p(s)ds
\]
uniformly exists with respect to $ t\in\R $. Notice that this last condition always holds when $p(t)$ is periodic.

The study of non-smooth differential equations provides insight into various phenomena in both natural and engineering domains, where systems exhibit sudden changes. This knowledge has practical applications in various fields, such as modeling mechanical systems that involve collisions (see, for instance, \cite{Brogliato1996,Jeffrey2018,Kowalczyk2008}) and biological and climate modeling that involves abrupt changes (see, for instance,  \cite{Anna2017,Carvalho2020,Leifeld2018,Piltz2014}). Particularly, the differential equation \eqref{e1} provides models of electronic circuit in the presence of a relay as noticed by \cite{Jacquemard2012}.

The equation \eqref{e1} is part of the family of the Duffing-type equations
\begin{equation}\label{int1}
	\ddot{x}+g(x)=p(t),
\end{equation}
for which, in the 1960's, Littlewood proposed  to investigate the boundedness of solutions by assuming that $p(t)$ is either bounded or periodic and that $ g $ (commonly called {\it saturation function}) satisfies a particular asymptotic condition (see \cite{Littlewood1968}). In \cite{Littlewood1966B}, Littlewood gave an example for which $p(t)$ is bounded and \eqref{int1} has an unbounded solution. In this context, Morris \cite{Morris1976} was the first one to provide an example in the conditions proposed by Littlewood for which all the solutions of \eqref{int1} are bounded, namely, $p(t)$ periodic continuous and $ g(x)=2x^{3} $. Soon after, Dieckerhoff and Zehnder \cite{Zehnder1987} extended the Morris result to the following more general family of equations:
\begin{equation*}\label{EqZehnder}
	\ddot{x}+x^{2n+1}+\sum_{j=0}^{2n}x^j p_j(t)=0, \quad n\geq 1,
\end{equation*}
with $p_j$ being periodic $\CC^{\infty} $ functions. The proof of the aforementioned results involves showing that there exist infinitely many invariant closed curves of the time-$T$ map, which in turn is equivalent to the existence of a sequence of nested invariant tori for the differential equation, and for this, the Moser's twisting map theorem is very useful (see \cite{Moser1962}).  Some results concerning the boundedness of solutions for non-smooth differential equations were also obtained by means of the variants of Moser's twisting map theorem (see, for instance, \cite{Kunze1997,Ortega1996}).

As mentioned in \cite{Li2001}, Ortega in a talk \cite{Ortega1998Talk} at Academia Sinica in 1998 suggested the question of whether all solutions of \eqref{int1}, when $ g(x)=\arctan(x) $ and $p(t)$ is periodic, are bounded or not. In this case, the saturation function is bounded and generates a small twist at infinity, which makes it difficult to apply the standard versions of the twisting map theorem. It fell to Li \cite{Li2001} to first answered this question, in the case that $p(t)$ is a $ \CC^{\infty} $ periodic function with vanishing average. In \cite{Wang2006} , Wang improved the result of Liu by considering $p(t)$ as a $ \CC^{5} $ periodic function with some smallness condition on its average. The non-smooth forced oscillator \eqref{e1} represents a limit scenario to the case introduced by Ortega in \cite{Ortega1998Talk}.

The goal of this paper is to provide a simple proof for the boundedness of every solutions of the differential equation \eqref{e1} in the case that $p(t)$ is a Lebesgue integrable periodic function with vanishing average. Despite of the vanishing average restriction, equation \eqref{e1} still presents a rich dynamics, for instance, periodic solutions \cite{Jacquemard2012} and chaotic behaviour \cite{Burra2020,Sprott2010}. Our reasoning is based on a simple constructive approach that allows us to prove the existence of a sequence of invariant tori such that the union of their interiors covers all the $(t,x,\dot x)$-space, $(t,x,\dot x)\in \s^1\times\R^2$. In addition, we will see that these tori are foliated by periodic solutions.

The paper is structured as follows. In Section \ref{MainResult}, we define some objects to be used throughout this work and we state our main result (Theorem \ref{ta}) concerning the existence of infinitely many invariant tori. Section \ref{fundlemma} is dedicated to provide a sufficient condition for the existence of an invariant torus of \eqref{s1}, while Section \ref{existence} is devoted to the proof of our main result.

\section{Statement of the main result} \label{MainResult}

In order to address the properties of the solutions of the differential equation \eqref{e1}, it is convenient to consider it written as the following first-order autonomous differential system in the extended phase space, by taking $ y=x' $, $t=\phi$ and $\phi'=1$
\begin{equation}\label{s1}
\begin{cases}
\phi'=1,\\
x'=y,\\
y'=-\sgn(x)+ p(\phi),
\end{cases}
\end{equation}
 $(\phi,x,y)\in \R\times\R^{2} $.  In addition, since $p(\phi)$ is $T$-periodic, the differential system \eqref{s1} can be considered in the quotient space  $(\phi,x,y)\in \s^1\times\R^{2} $, where $\s^{1} =\R/T\mathbb{Z}$. In this way, an invariant torus of \eqref{s1} in the quotient space $\s^{1}\times\R^{2}$  corresponds to an invariant cylinder of \eqref{s1} in the extended phase space $\R\times\R^{2}$, such that its intersections with the ``time'' sections $\Lambda_{0}:=\{0\}\times\R^2$ and $\Lambda_{T}:=\{T\}\times\R^2$ coincide in the quotient space.

 \subsection{Existence and uniqueness of solutions}\label{sec:eus}
 
The differential system \eqref{s1} has two kinds of discontinuities, namely, the ones generated by  the $\sgn$ function and the ones possibly generated by the Lebesgue integrable function $p(t)$. Before stating our main result, some comments about  the existence and uniqueness of solution of \eqref{e2} are necessary.

By taking $ y=x' $, the differential equation \eqref{e1} can be written as the following first-order differential system:
\begin{equation}\label{e2}
\begin{cases}
x'=y,\\
y'=-\sgn(x)+ p(t).
\end{cases}
\end{equation}
 Filippov convention will be assumed for solutions of the differential system \eqref{e2} (see \cite[\S 7]{Filippov1988}), which exist for every initial conditions (see \cite[Theorem 8]{Filippov1988}). The differential system \eqref{e2} will be referred by Filippov system. As usual, solutions of the differential system \eqref{s1} are given in terms of solutions of the Filippov system \eqref{e2} and, therefore, also exist for every initial conditions.

The solutions of \eqref{e2} can be investigated by considering the following differential systems:
\begin{equation}\label{e2lateral}
\begin{cases}
x'=y,\\
y'=-1+ p(t),
\end{cases} x\geq 0, \quad\text{and}\qquad
 \begin{cases}
x'=y,\\
y'=1+ p(t),
\end{cases} x\leq0,
\end{equation}
which match  \eqref{e2} restricted to $x\geq0$ and $x\leq0$, respectively. Since $p(t)$ is a Lebesgue integrable function,   the differential systems in \eqref{e2lateral} correspond to Carathéodory differential systems  for which all the conditions for the existence and uniqueness of solutions are satisfied  (see \cite[Chapter 18]{Kurzweil2014} for more details).
In the extended phase space, such differential systems become:
\begin{equation}\label{s1lateral}
\begin{cases}
\phi'=1,\\
x'=y,\\
y'=-1+ p(\phi),
\end{cases} x\geq 0, \quad\text{and}\qquad
 \begin{cases}
 \phi'=1,\\
x'=y,\\
y'=1+ p(\phi).
\end{cases}x\leq 0.
\end{equation}
Due to the presence of the $ \sgn $ function, the switching plane $ \Sigma=\{(\phi,x,y)\in\R\times\R^{2}:x=0\} $ plays  an important role in describing the  dynamics of \eqref{s1}.
 Indeed, one can see that, for each initial condition $(\phi_0,0,y_0)\in\Sigma$ with $y_0\neq0$, the unique maximal solutions of the differential systems in \eqref{s1lateral} are transversal to $\Sigma$ at $(\phi_0,0,y_0)$ and concatenate in order  to  form a (local) solution of the differential system \eqref{s1} that is unique around $(\phi_0,0,y_0)$.  Such a solution could eventually lose its uniqueness property if it reaches $\Sigma$ tangentially, that is, at a point  $(\phi_1,0,0)$. However, as long as a maximal solution of the differential system  \eqref{s1} is transversal to $\Sigma$, it is given by the concatenation of solutions of the differential systems in \eqref{s1lateral} providing, then, the uniqueness of such a maximal solution, which has the whole line as its interval of definition. The explicit expressions of the solutions of the differential systems in \eqref{s1lateral}  will be provided below in Section \ref{fundlemma}.

 \subsection{Main result}
 
 Let us define
\begin{equation*}\label{primitives}
	  P_1(t):=\int_{0}^{t}p(s)\d s \quad  \text{and} \quad P_2(t):=\int_{0}^{t}P_1(s)\d s,
\end{equation*} 
and, as usual, let $\overline{p}$ denote the average of $p(t)$, i.e.
\[
\overline{p}:=\frac{1}{T}\int_{0}^{T}p(s)\d s=\dfrac{P_1(T)}{T}.
\]
Notice that the function $P_1(t)$ is continuous and the function $P_2(t)$ is continuously differentiable.

Our main result provides the existence of a sequence of invariant tori for \eqref{s1} under the assumption $\ov p=0$.

\begin{mtheorem}\label{ta}
Suppose that $p(t)$ is a Lebesgue integrable $T$-periodic function satisfying $\ov p=0$. Then, there exists a sequence $\Tn\subset\mathbb{S}^1\times\R^2$ of nested invariant tori of the differential system \eqref{s1}  foliated by periodic solutions and satisfying:
\[
\mathbb{S}^1\times\R^2=\bigcup_{n\in\N} \inte(\Tn),
\]
where $ \inte(\Tn) $ denotes the open region enclosed by $\Tn$.
In addition, all the maximal solutions of \eqref{s1} are defined for every $t\in\R$ and the ones starting at $(\mathbb{S}^1\times\R^2)\setminus\inte(\mathbb{T}_1)$ are unique and transversal to $\Sigma$.
\end{mtheorem}

Since $ \Tn $ is invariant for each $ n\in\N$, the uniqueness property provided by Theorem \ref{ta} implies that a solution starting at $ \ov{ \mbox{int}(\Tn)}$ cannot leave it for all $ t\in\R $. This leads us to the following corollary.

\begin{corollary}\label{c1}
Suppose that $p(t)$ is a Lebesgue integrable $T$-periodic function satisfying  $\ov p=0$.  Then, all the solutions of the differential system \eqref{e2} are bounded. 
\end{corollary}

\section{Fundamental Lemma}\label{fundlemma}

This section is devoted to provide a sufficient condition for the existence of invariant tori of \eqref{s1}.

For each $n\in \N$, define the functions $y_n^+:[0,T]\to\R$ and $y_n^-:[0,T]\to\R$ by
 \begin{equation}\label{eq:yn}
 	y_n^{\pm}(\phi_0)=\pm\frac{nT}{2}+P_1(\phi_0)-\frac{P_2(T)}{T}
 \end{equation}
and, for each $n\in\N,$ such that $y_n^-(\phi_0)<y_n^+(\phi_0)$ for every $\phi_0\in[0,T],$ define the surface
\begin{equation}\label{inv.cylinders}
\begin{aligned}
	&\mathcal{T}_n:=	\mathcal{T}^{+}_n\cup	\mathcal{T}^{-}_n,\,\,\text{ where}\\
&\mathcal{T}^{\pm}_n:=\{(\phi_0,\Psi^{\pm}_n(\phi_0,y_0),y_0): \phi_0\in\R, \; y_0\in[y^{-}_n(\phi_0),y^{+}_n(\phi_0)] \},
\end{aligned}
\end{equation}
and
	\begin{align*}
		\Psi^{\pm}_n(\phi_0& ,y_0):=\frac{1}{8}\left(\pm n^{2} T^{2} \mp4 y_0^{2}-8P_2\left(\frac{n
			T}{2}\pm y_0\mp P_1(\phi_0)\pm \frac{P_2(T)}{T}+\phi_0\right)\right.\\
		&\hspace{1.5cm}+
\left.	 4 P_2(T) \left(n\pm\frac{P_2(T)}{T^{2}}\right)-4
		P_1(\phi_0) (\pm P_1(\phi_0)\mp2 y_0)+
		8 P_2(\phi_0)\right).
	\end{align*}

The following result provides sufficient conditions for which the surface $ \mathcal{T}_n$, for some $n\in\N$, corresponds to an invariant torus of  \eqref{s1}.
\begin{lemma}[Fundamental Lemma]\label{fundamentallemma} Let $n\in\N$ be fixed. Assume that, for every $\phi_0\in[0,T]$, 
\begin{equation}\label{cond1}
\big|T P_1(\phi_0)-P_2(T)\big|<\dfrac{n T^2}{2} 
\end{equation}
and
\begin{equation}\label{cond2}
\big|t P_2(T)+T P_2(\phi_0)- T P_2(t+\phi_0)\big|<\dfrac{T}{2}t(nT-t),\,\, t\in(0,n T).
\end{equation}
Then, $\mathcal{T}_n$ is an invariant torus of the differential system \eqref{s1}. Moreover, $ \mathcal{T}_n $ is foliated by $2nT$-periodic solutions.
\end{lemma}

Before we proceed with the proof of Lemma \ref{fundamentallemma}, we define some important objects to help us along the process. Consider the following sequence of curves in the plane $ \Sigma$:
\begin{equation}\label{gammapn}
	\gamma^{+}_n= \{(\phi_0,0,y_n^{+}(\phi_0)):\phi_0\in\mathbb{S}^{1}\}\,\, \text{and} \,\, \gamma^{-}_n= \{(\phi_0,0,y_n^{-}(\phi_0)):\phi_0\in\mathbb{S}^{1}\},
\end{equation}
where $y^{+}_n$ and $y^{-}_n$ are the continuous functions defined in~\eqref{eq:yn}. 

We observe that the solutions of the  differential systems in \eqref{s1lateral} are given, respectively, by 
\begin{equation*}\label{solp}
	\begin{aligned}
		\varphi^{+}(t;\phi_0,x_0,y_0)=\Big(&t + \phi_0, -\frac{t^{2}}{2} +x_0 + t y_0 - t P_1(\phi_0) - P_2(\phi_0) + 
		P_2(t + \phi_0) ,\\
		& -t + y_0 - P_1(\phi_0 )+ P_1(t + \phi_0)\Big)
	\end{aligned}
\end{equation*}
and
\begin{equation*}\label{soln}
	\begin{aligned}
		\varphi^{-}(t;\phi_0,x_0,y_0)=\Big(&t + \phi_0,\frac{t^{2}}{2}+x_0 + t y_0 - t P_1(\phi_0) - P_2(\phi_0) + 
		P_2(t + \phi_0)  ,\\
		& t + y_0 - P_1(\phi_0 )+ P_1(t + \phi_0)\Big),
	\end{aligned}
\end{equation*}
which are obtained via  direct  integration. As mentioned in Section \ref{sec:eus}, if a solution $\varphi(t)$ of the differential system \eqref{s1} is transversal to $\Sigma$, then it writes as concatenations of $\varphi^{+}$ and $\varphi^{-}$.

 Also, the next remark concerning the primitives $ P_1 $ and $ P_2 $ plays an important role throughout this work.
\begin{remark}\label{identities}
	The following identities hold for every $ t\in\R $ and $ n\in\N $:
	\begin{equation*}\label{id1}
		P_1(t+nT)=P_1(t)+n P_1(T) 
	\end{equation*}
	and
	\begin{equation*}\label{id2}
		P_2(t+nT)=P_2(t)+n P_1(T)t+\dfrac{n^{2}-n}{2}T P_1(T)+n P_2(T).
	\end{equation*}
	In particular, if $p(t)$ has vanishing average, then $ P_1 $ is $ T$-periodic.
\end{remark}

\subsection{Proof of the Fundamental Lemma}	
Let us first prove that $ \varphi^{\pm} $ takes the curves $ \gamma^{\pm}_n$ into $ \gamma^{\mp}_n $ (see Figure \ref{fig1}), that is
\begin{equation}\label{rel1}
\varphi^{\pm}(t;\phi_0,0,y^{\pm}_n(\phi_0)) \notin \Sigma\,\,\, \forall t\in(0,nT)
\end{equation}
and
\begin{equation}\label{rel2}
\varphi^{\pm}(nT;\phi_0,0,y^{\pm}_n(\phi_0))=(nT+\phi_0,0,y^{\mp}_n(\phi_0)).
\end{equation}
Indeed, condition \eqref{cond1} implies that
	\[
		y^{+}_{n}(\phi_0)>0 \;\;\; \text{and}\;\;\; 	y^{-}_{n}(\phi_0)<0,
	\]
which means that the points in $\gamma^{+}_n$ and $ \gamma^{-}_n$ follows the forward flows $ \varphi^{+} $ and $ \varphi^{-} $, respectively. From now on, $ \varphi_i^{\pm} $, $ i=1,2,3, $ denotes the $ i $-th coordinate of the function $ \varphi^{\pm}$, thus condition \eqref{cond2} implies that
	\[
 \varphi^{+}_2(t;\phi_0,0,y^{+}_n(\phi_0))>0\;\;\; \text{and}\;\;\;  \varphi^{-}_2(t;\phi_0,0,y^{-}_n(\phi_0))<0,
	\]
for every $ t\in(0,nT) $, which provides that relationship \eqref{rel1} holds. Moreover, taking Remark \ref{identities} into account, we have
\begin{align*}
	\varphi^{\pm}_2 (nT;\phi_0,0,y^{\pm}_n(\phi_0))&=\mp\dfrac{(nT)^{2}}{2} + nT\; y^{\pm}_n(\phi_0) - nT \;P_1(\phi_0) - P_2(\phi_0) + 	P_2(nT + \phi_0)\\
	&=P_2(nT + \phi_0)- P_2(\phi_0)-n P_2(T)\\
	&=P_1(T)(n(\phi_0+\dfrac{n-1}{2})	)\\
	&=0, 
\end{align*}
because $ P_1(T)=0 $. Also,
\[
\begin{aligned}
	\varphi^{\pm}_3 (nT;\phi_0,0,y^{\pm}_n(\phi_0))&=\mp nT +y^{\pm}_n(\phi_0) -P_1(\phi_0)+P_1(nT + \phi_0)\\
	&=\mp nT \pm\dfrac{nT}{2} +P_1(\phi_0) - \dfrac{P_2(T)}{T} -P_1(\phi_0)+P_1(nT + \phi_0)\\
	&=\mp\dfrac{nT}{2} +P_1(\phi_0) - \dfrac{	P_2(T)}{T} \\
	&=y^{\mp}_n(\phi_0).
\end{aligned}
\]
Hence, relationship \eqref{rel2} holds.

\begin{figure}[h]
	\begin{center}
	\begin{overpic}[scale=0.4]{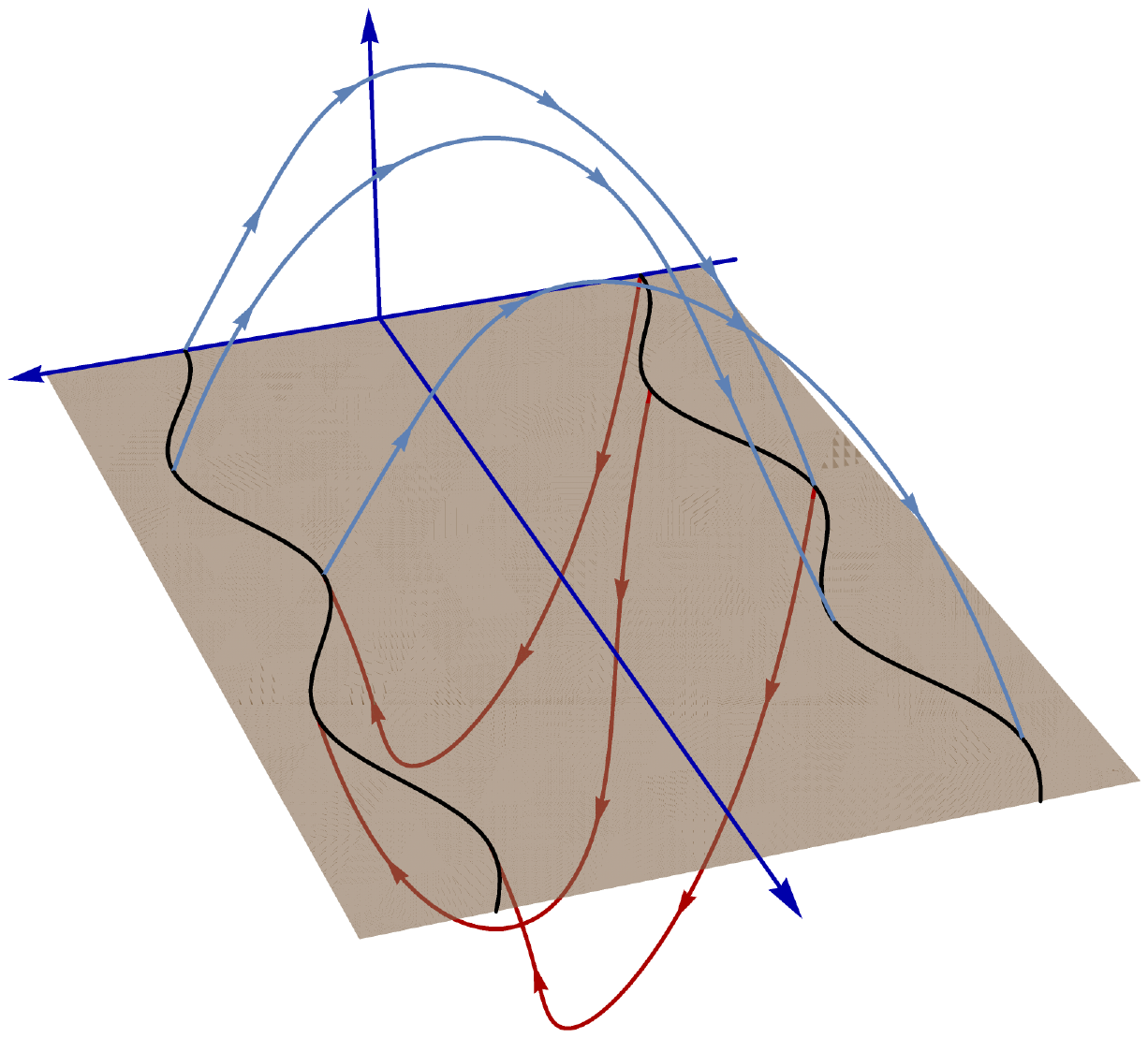}
			\put(27,17){{ $ \Sigma $}}
			\put(44.5,16){$ \gamma^{+}_{n} $}
			\put(76,16){$ \gamma^{-}_{n} $}
			\put(40,67){{ $ \varphi^+ $}}
						\put(47,2){{ $ \varphi^- $}}
			\put(63.5,9){$ \phi $}
			\put(35.5,69.5){$ x $}
			\put(11.5,44.5){$ y $}
		\end{overpic}		
	\end{center}
	\caption{For $ n\in\N $ satisfying \eqref{cond1} and \eqref{cond2} , the flow $ \varphi^{\pm} $ takes $ \gamma^{\pm}_n $ into $ \gamma^{\mp}_n $ .}
	\label{fig1}
	\smallskip
\end{figure}

Accordingly, we notice that
\begin{equation}\label{inv.surfaces}
	\begin{aligned}
		&\mathcal{S}_n:=	\mathcal{S}^{+}_n\cup	\mathcal{S}^{-}_n,\,\,\text{ where}\\
		&\mathcal{S}^{\pm}_n:=\{\varphi^{\pm}(t;\phi_0,0,y^{\pm}_n(\phi_0)): t\in[0,nT], \; \phi_0\in\R\},
	\end{aligned}
\end{equation}
is an invariant surface of \eqref{s1} whose intersections with $ \Sigma $ correspond to  the curves $\gamma^{+}_n$ and $\gamma^{-}_n$ given by \eqref{gammapn} (see Figure \ref{fig1}). Now, by solving the system of equations
\begin{equation}\label{sistema}
	\begin{aligned}
		\varphi^{\pm}_1(t;\phi,0,y^{\pm}_n(\phi))&=\phi_0,\\
		\varphi^{\pm}_3(t;\phi,0,y^{\pm}_n(\phi))&=y_0,
	\end{aligned}
\end{equation}
in the variables $ (t,\phi) $ and, then, substituting the solutions into $\varphi^{\pm}_2(t;\phi,0,y^{\pm}_n(\phi))$, we get that $ \mathcal{S}_n= \mathcal{T}_n$. Furthermore, $ \mathcal{T}_n^{+}=\mathcal{T}_n\cap \{x\geq0\}$ and $\mathcal{T}_n^{-}=\mathcal{T}_n\cap \{x\leq0\}$ are homeomorphic to 
$$ \mathcal{D}_n:= \{(\phi_0,y_0):\phi_0\in\R, \;  y_0\in[y^{-}_n(\phi_0),y^{+}_n(\phi_0)]\},$$ 
because they are graphs of $ \Psi^{+}_n $ and $ \Psi^{-}_n $, respectively. This in turn implies that $ \mathcal{T}_n $ is a simply connected surface and, consequently, $\mathcal{T}_n$ is an  invariant  cylinder of \eqref{s1}.

Now, let $U\subset \Lambda_0$ be the set of initial conditions in $\Lambda_0$ for which the corresponding maximal solutions of \eqref{s1} are transversal to $\Sigma$. As mentioned before, such solutions are unique and defined for every $t\in\R$. Thus, consider the time-$ T $-map $\CP_{T}$ defined on $U$ into $\Lambda_{T}$:
\begin{equation*}\label{poincare}
	\begin{array}{rccc}
		\CP_{T}: &U&\longrightarrow&	\Lambda_{T}\\
	&(0,x_0,y_0)&\longmapsto &\varphi(T;0,x_0,y_0).
	\end{array} 
\end{equation*}
Since $\Lambda_{0}$ and $\Lambda_{T}$ coincide in the quotient space $ \s^{1}\times\R^{2}$ and taking into account that $\CP_{T}(U)\subset \Lambda_T$ corresponds to the set of initial  conditions in $\Lambda_T$  for which the maximal solutions of \eqref{s1} are transversal to $\Sigma$, it follows that $\CP_{T}$ can be seen as an automorphism on $U$.

Let $\CC_n^{0}$ and $\CC_n^{T}$ denote the intersections between the invariant cylinder $\mathcal{T}_n$ with the time sections $ \Lambda_{0} $ and $ \Lambda_{T} $, respectively. Notice that, from the considerations above, $\CC_n^{0}\subset U$. In addition, 
\begin{equation}\label{C0}
	\begin{aligned}
		\CC_n^{0}:=&\mathcal{T}_n\cap\Lambda_{0}=\CC^{+}_{n,0}\cup\CC^{-}_{n,0},\,\,\text{ where}\\
		\CC^{\pm}_{n,0}:=&\{(0,\Psi^{\pm}_n(0,y_0),y_0):y_0\in[y^{-}_n(0),y^{+}_n(0)] \}
	\end{aligned}
\end{equation}
and
\begin{equation}\label{CT}
	\begin{aligned}
		\CC_n^{T}:=&\mathcal{T}_n\cap\Lambda_{T}=\CC^{+}_{n,T}\cup\CC^{-}_{n,T},\,\, \text{ where}\\
		\CC^{\pm}_{n,T}:=&\{(T,\Psi^{\pm}_n(T,y_0),y_0):y_0\in[y^{-}_n(T),y^{+}_n(T)] \}.	
	\end{aligned}
\end{equation}

 In what follows, we show  that $ \CC_n^{0} $ is  invariant under the map $  \CP_{T} $. For this purpose, we will examine the parametrizations of $\CC_n^{0}$ and $\CC_n^{T}$ given by \eqref{C0} and \eqref{CT}, respectively. We first observe that the functions $y^{+}_n$ and $y^{-}_n$ are $ T $-periodic, thus $ \Psi^{\pm}_n(0,\cdot) $ and $ \Psi^{\pm}_n(T,\cdot)  $  have the same domain, namely, $\mathcal{I}_n=[y^{-}_n(0),y^{+}_n(0)]$. Then, it is sufficient to show that
	\[
	\Psi^{\pm}_n(0,y_0)-\Psi^{\pm}_n(T,y_0)=0  \quad \mbox{for all} \quad y_0\in \mathcal{I}_n.
	\]
	In fact,
	\[
	\begin{aligned}
		\Psi^{\pm}_n(0,y_0)-\Psi^{\pm}_n(T,y_0)= &n\dfrac{P_2(T)}{2} -P_2\left(\dfrac{nT}{2}\pm\dfrac{P_2(T)}{T}\pm y_0\right)\\
		&+\left(
		P_2\left(T+\dfrac{n
			T}{2} \pm\dfrac{P_2(T)}{T
		}\pm y_0 \right)-(n+2)\dfrac{P_2(T)}{2}
		\right)\\
		=&n\dfrac{P_2(T)}{2}+P_2(T)-(n+2)\dfrac{P_2(T)}{2}\\
		=&0.
	\end{aligned}
	\]
	The second equality above was obtained by the identity given in Remark \ref{identities}. Hence, it follows that the invariant cylinder $\mathcal{T}_n$ corresponds to an invariant torus of \eqref{s1} in the quotient space $ \s^{1}\times\R^{2}$ (see Figure \ref{fig3}), which concludes this proof.
	
	\begin{figure}[h]
		\begin{center}
	\begin{overpic}[scale=0.4]{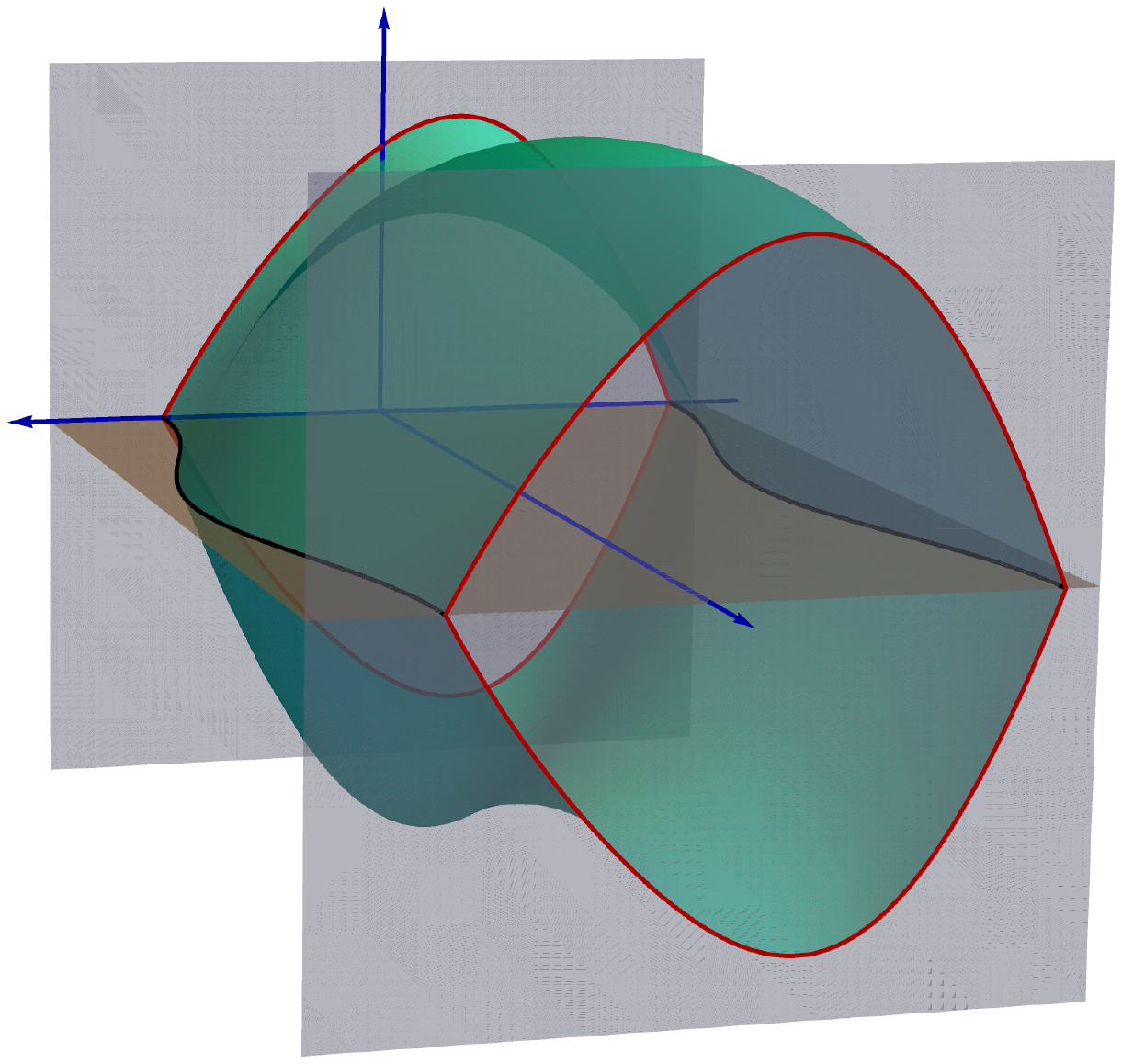}
				\put(16,37){ $ \Sigma $}
				\put(12,17){ $ \Lambda_0 $}
				\put(28,-1){ $ \Lambda_T $}
				\put(27,18){ $ \mathcal{T}_n^{-} $}
				\put(55,60){ $ \mathcal{T}_{n} ^{+}$}
				\put(68.5,13){ $ \mathcal{C}_{n}^T $}
				\put(24,52.5){$\mathcal{C}_n^0$}
				\put(34.5,69.5){$x$}
				\put(9,42){$y$}
				\put(59,27){$\phi$}
			\end{overpic}		
		\end{center}
		\caption{Invariant torus $\mathcal{T}_n=\mathcal{T}_n^+\cup \mathcal{T}_n^-$ provided by the Fundamental Lemma.}
		\label{fig3}
		\smallskip
	\end{figure}

\section{Proof of Theorem \ref{ta}}\label{existence}

The proof of Theorem \ref{ta} will follow as a consequence of the next results. The first one, Proposition \ref{fundprop}, will provide the existence of $ n^{*}\in\N $ such that the conditions \eqref{cond1} and \eqref{cond2} of the Fundamental Lemma \ref{fundamentallemma} are satisfied for every $ n\geq n^{*}$. Accordingly, the sequence of invariant tori stated by Theorem \ref{ta} will be given by $ \Tn:=\mathcal{T}_{n+n^{*} }$, $n\in\N$. Finally, Corollary \ref{cor:trans-unique} will provide that each maximal solution of the differential system \eqref{s1} is defined for  every $t\in\R$ and the ones starting in $ (\mathbb{S}^1\times\R^2)\setminus\inte(\mathbb{T}_1)$ are transversal to $\Sigma$ and, consequently, unique.

\begin{proposition}\label{fundprop}
	Let $ p (t)$ be a Lebesgue integrable $ T $-periodic function such that $\ov p=0$. Then, there exists $ n^{*} \in\N$ such that $ \mathcal{T}_n $ is an invariant torus of \eqref{s1} for every $ n \geq n^{*} $.
\end{proposition}

\begin{proof}
	From Remark \ref{identities} we have that $ TP_1(\phi_0)-P_2(T) $ is continuous $ T $- periodic in $ \phi_0 $ and consequently bounded, thus there exists $ n_0\in\N $ such that the relationship \eqref{cond1} holds  for every $ n\geq n_0 $. 

In order to obtain \eqref{cond2}, we define the functions 
	\begin{equation}\label{curveshn}
	\begin{aligned}
		f(t):=&t P_2(T)+T P_2(\phi_0)- T P_2(t+\phi_0) \,\, \text{and} \\
		 h_n(t):=&\dfrac{T}{2}t(nT-t),\,\, \text{for}\,\, n\in\N.
		\end{aligned}
	\end{equation}
By Remark \ref{identities}, we notice that $ f $ is continuously differentiable and $ T $-periodic. Besides that, $ h_n(t)>0 $ for $ t\in(0,nT) $. 

We  start by  proving the following claim:

\smallskip

\noindent\textbf{Claim 1.} {\it There exists $ n^{*}\geq n_0 $ such that 
\begin{equation}\label{claim1}
	|f(t)|<h_{n^{*}}(t) \;\;\;  \forall t\in(0,n^{*}T).
\end{equation} 
}

\smallskip

Consider the functions
\[
q_n^{+}(t):=f(t)-h_n(t)\;\;\; \text{and} \;\;\; q_n^{-}(t):=f(t)+h_n(t).
\]
Since,  $ q_{n_0}^{+}(nT) =q_{n_0}^{-}(nT)=0$ , $ (q_{n_0}^{+} ) '(n_0T)>0 $ and $ (q_{n_0}^{-} ) '(n_0T)<0 $, then there exists $ \e_0>0 $ such that $q_{n_0}^+(t)>0$ and $q_{n_0}^-(t)<0$ for every $t\in(n_0T-\e_0,n_0T)$. This implies that
\begin{equation}\label{eps0}
	|f(t)|<h_{n_0}(t) \;\;\; \forall t\in(n_0T-\e_0,n_0T).
\end{equation}

Now, for each $k\in\N$, consider the following $(2^k n_0 T)$-periodic function: \begin{equation}\label{eq:hext}
	\ov h _k(t):=\sum_{m\in\N}\chi_{I_m}(t)h_{2^{k}n_0}(t-2^{k}(m-1)T),
\end{equation}
where $ I_m:=\left.\left[2^{k}(m-1)T,2^{k}mT\right.\right) $ and $ \chi_{I_m} $ is the characteristic function of $  I_m $. Notice that $\ov h _0(t)$ is a $n_0 T$-periodic extension of $ h_{n_0}$ and that $ \ov h _{k+1}(t)\geq\ov h _{k}(t)  $ for every $t \geq 0$ and  $k\in\N$ (see Figure \ref{curveshk}). Taking \eqref{eps0} into account, this implies that 
\begin{equation}\label{gen}
		|f(t)|<	\ov h _k(t) \;\;\; \forall t\in(2^{k}n_0T-\e_0,2^{k}n_0T),
\end{equation}
for every $ k\in\N$.
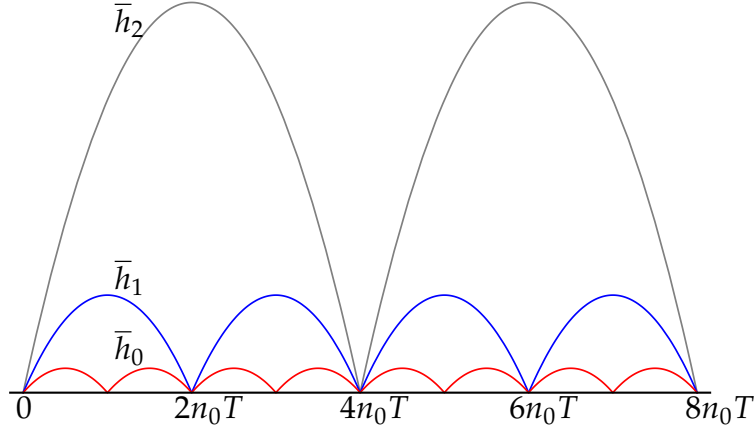
\begin{figure}[h]
	\begin{center}
\begin{overpic}[scale=0.4]{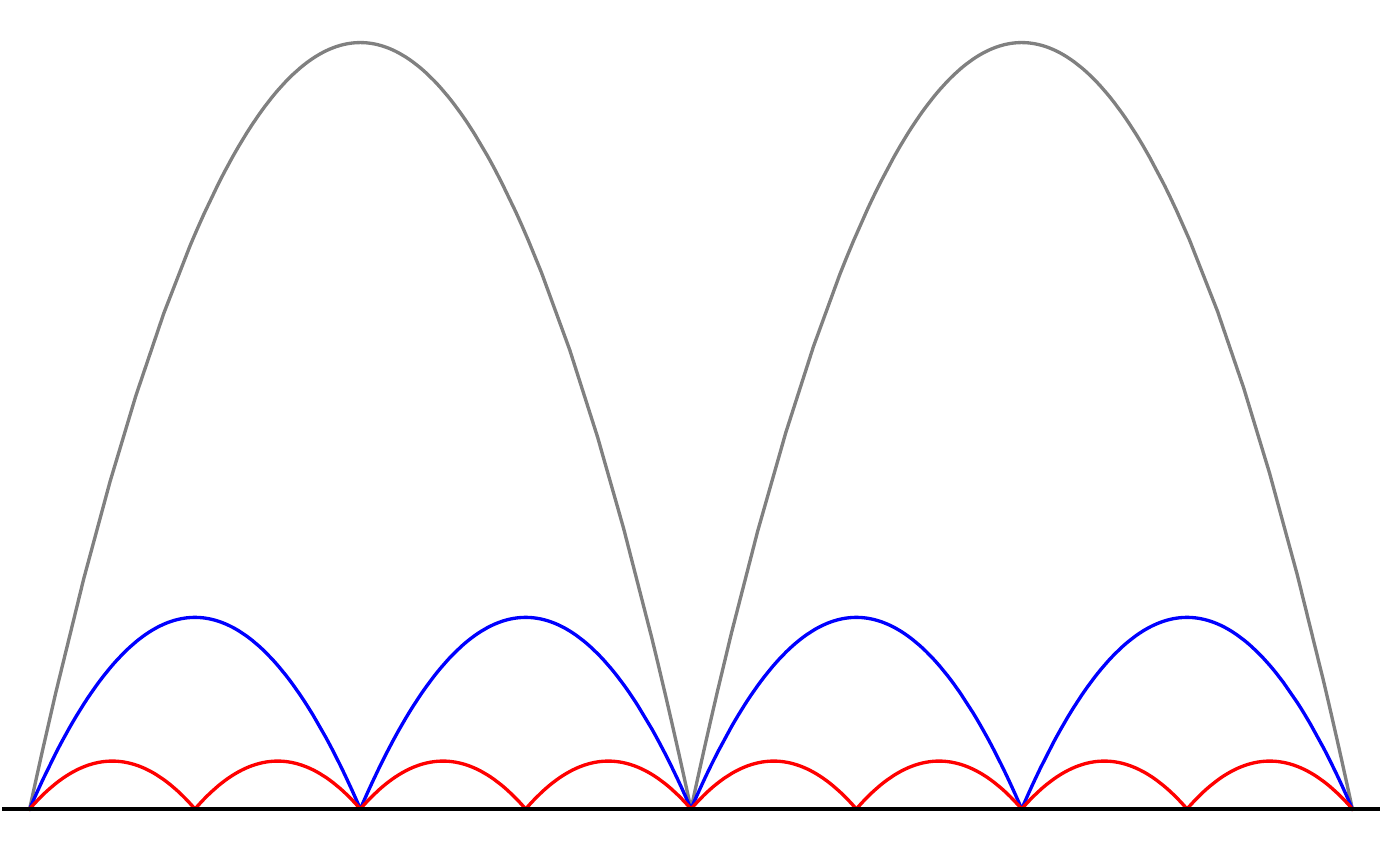}
			\put(14,8){{ $ \ov h _0 $}}
			\put(14,18){{ $ \ov h _1 $}}
			\put(14,54.5){{ $ \ov h _2 $}}
			\put(0,-0.4){{ $0$}}
			\put(22.5,-0.4){{ $ 2n_0 T$}}
			\put(46,-0.4){{ $ 4n_0 T$}}
				\put(70,-0.4){{ $ 6n_0 T$}}
			\put(95,-0.4){{ $ 8n_0 T$}}
		\end{overpic}		
	\end{center}
	\caption{Graphs of the functions $ \ov h _k $ constructed in \eqref{eq:hext} for $ k=0,1,2 $.}
	\label{curveshk}
	\smallskip
\end{figure}

 On the other hand, from Remark \ref{identities}, it follows that $ P_1 $ is bounded, so let $ M>0 $ satisfy $ \|P_1\|_{\infty}=M $. Then, 
\begin{equation}\label{boundf1}
	|f(t)|\leq t\int_{0}^{T}|P_1(s)|\d s +T\int_{0}^{t}|P_1(s+\phi_0)|\d s\leq 2 T M t, \,\,\forall\,\, t\in\R.
\end{equation}
Consequently,  
\begin{equation}\label{tk}
	|f(t)|<\ov h _{k}(t) \; \; \; \text{for all} \;\;\; t\in(0,2^{k}n_0T-4M).
\end{equation}
In addition, since $f(t)$ is $T$-periodic, \eqref{boundf1} also implies that
\begin{equation}\label{boundf2}
|f(t)|\leq 2T^2M\,\,\text{ for every }\,\, t\in\R.
\end{equation}

Assume, by reduction to absurdity, that \eqref{claim1} does not hold. In particular, for each $ k\in\N $, there exists $ t_k\in(0,2^{k}n_0T) $ such that  $|f(t_k)|=\ov h _{k}(t_k).$
 From \eqref{tk}, it follows that $ t_k\in[2^{k}n_0T-4M,2^{k}n_0T) $, which means that $ t_k\to\infty $ when ${ k\to\infty }$. Moreover, from \eqref{boundf2}, one has
\[
|t_k-2^{k}n_0T|=\dfrac{\ov h _{k}(t_k)}{|t_k|}=\dfrac{|f(t_k)|}{|t_k|}\leq \dfrac{2T^{2}M}{|t_k|}\to 0, \;\;\; \text{when} \;\;\; k\to \infty.
\]
Then, there exists $ k_0 \in\N$ such that
$ |t_k-2^{k}n_0T|<\e_0 $ for every $ k\geq k_0 $, with $ \e_0 >0$ being the one satisfying \eqref{eps0}. This is contradiction with \eqref{gen}. 

Thus, there must exists $ k^{*}\in\N$, for which $ |f(t)|<	\ov h _{k^{*}}(t) $ for all $ t\in(0,2^{k^{*}}n_0T) $. Hence, Claim 1 follows by defining $ {n^{*}=2^{k^{*}}n_0 }$.

Finally, the proof of the proposition will follow by proving the  next claim:

\bigskip

\noindent\textbf{Claim 2.} {\it  Condition \eqref{cond2} holds for every $ n\geq n^{*} $.}

\smallskip

We notice that 
\[
|f(t)|< h _{n^{*}}(t-T) \;\;\; \text{for all} \;\;\; t\in(T,(n^{*}+1)T).
\]
Taking into account that $h_{n^{*}}(t)$ and $h_{n^{*}}(t-T)$ coincide for $t=(n^{*}+1)T/2$, we can define
\[
\tilde h_{n^{*}}(t):=\left\{\begin{aligned}
h _{n^{*}}(t) \;\; &\text{if} \;\; t\in[0,\tfrac{(n^{*}+1)T}{2}],\\
h _{n^{*}}(t-T) \;\; &\text{if} \;\; t\in(\tfrac{(n^{*}+1)T}{2},(n^{*}+1)T].
\end{aligned}\right.
\]

Now, since $  h _{n^{*}}(t)\leq  h _{n^{*}+1}(t) $ for all $ t\in(0,n^{*}T) $ and $  h _{n^{*}}(t-
T)\leq  h _{n^{*}+1}(t) $ for all $ t\in(T,(n^{*}+1)T) $, we have that
$ |f(t)|< \tilde h_{n^{*}}(t)\leq h_{n^{*}+1}(t)$ for all $ t\in(0,(n^{*}+1)T) $. The proof of the claim follows, then, by repeating this procedure recursively. 

\smallskip

Hence, Lemma \ref{fundamentallemma} ensures that, for each $ n\geq n^{*} $, $\mathcal{T}_n $ is an invariant torus of \eqref{s1}. 
\end{proof}

\begin{corollary}\label{cor:trans-unique}
	Let $p(t)$ be a Lebesgue integrable function with vanishing average. Then, all the maximal solutions of \eqref{s1} are defined for every $t\in\R$ and the ones whose initial conditions lie on $(\mathbb{S}^1\times\R^2)\setminus\inte(\mathbb{T}_1)$ are unique and transversal to $\Sigma$.
\end{corollary}

\begin{proof}
We start by proving that, for each $(\phi_0,x_0,y_0)\in(\mathbb{S}^1\times\R^2)\setminus\inte(\mathbb{T}_1)$, there exists a unique maximal solution passing through $(\phi_0,x_0,y_0)$ which is transversal to $\Sigma$ and it is defined for  every $t\in\R$.
 
Notice that a maximal solution with initial condition $(\phi_0,x_0,y_0)$ may intersect the plane $\Sigma$ or not. If such an intersection does not occur, then such a solution is unique as a maximal solution of one the differential systems in \eqref{s1lateral}.  On the other hand, if $\varphi(t;\phi_0,x_0,y_0)$ intersects $\Sigma$, then such an intersection must be transversal. Otherwise, there would exist a time $t^{*}>0$  such that
 \[
 \varphi_2(t^{*};\phi_0,x_0,y_0)=0 \quad \text{and} \quad \varphi_3(t^{*};\phi_0,x_0,y_0)=\frac{\d}{\d t}\varphi_2(t^{*};\phi_0,x_0,y_0)=0,
 \]
which implies that such a solution would cross $\mathbb{T}_1$ (from outside to inside) contradicting the fact that $\mathbb{T}_1 $ is invariant and that the solutions there defined are unique. Thus, as noticed in Section \ref{sec:eus}, this transversality implies the uniqueness of all solutions starting in $(\mathbb{S}^1\times\R^2)\setminus\inte(\mathbb{T}_1)$ which, therefore, are defined for every $t\in\R$.

For the remaining initial conditions, as a consequence of the invariance of $\mathbb{T}_1$, the maximal solutions starting in the set $\inte(\mathbb{T}_1)$ are confined in the compact set $\ov{\inte(\mathbb{T}_1)}$. Therefore, they must be defined for all $t\in\R$ (see \cite[Theorem 9 \S 7]{Filippov1988}). This concludes the proof of the corollary.
\end{proof}

\subsection{A simpler approach for $L^{\infty}$-forcing term}

In the next result, we shall see that the proof of Proposition \ref{fundprop} simplifies a lot by assuming $p$ to be an $ L^{\infty} $-function on $ [0,T] $, instead of just Lebesgue integrable. In this case, we will show that the surface $ \mathcal{T}_n $ provided in \eqref{inv.cylinders} is an invariant torus of \eqref{s1} for every $ n\in\N $ bigger than $ 2 \|p\|_{L^{\infty}} $.
\begin{proposition}
	Let $  p $ be a $ T $-periodic function with vanishing average and suppose that there exists $ M>0 $ such that ${ \|p\|_{L^{\infty}}<M }$. Then, the surface $ \mathcal{T}_n $ is an invariant torus of \eqref{s1} for all $ n\in\N $ satisfying $ n\geq 2 M $.
\end{proposition}

\begin{proof}
	We recall that in order to obtain this result, it is sufficient to show that conditions \eqref{cond1} and \eqref{cond2} hold for all $ n\in\N $ such that $ n\geq 2 M $. 
	
Define $\alpha(\phi_0):=T P_1(\phi_0)-P_2(T)$. Notice that $ \alpha $ is $ T $-periodic by Remark \ref{identities}, which restricts our analysis to $ \phi_0\in[0,T] $. Then, taking into account that $ \|p\|_{L^{\infty}}<M $, we see that, for $ \phi_0\in[0,T] $,
	\begin{equation}\label{alpha}
	\begin{aligned}
		|\alpha(\phi_0)|&=\left|\int_{0}^{T}P_1(\phi_0)-P_1(t)\d t\right|\\
		&\leq \int_{0}^{T}\|p\|_{L^{\infty}}|\phi_0-t|\d t\\
		&<M\int_{0}^{T}|\phi_0-t|\d t=M\left(\frac{T^{2}}{2}- T \phi_0  +\phi_0^{2}\right)\leq M\dfrac{ T^2}{2}\leq \dfrac{ n T^2}{4},
	\end{aligned}
	\end{equation}
whenever $ n\in\N $ satisfies $ n\geq 2M $. Therefore, condition \eqref{cond1} holds for every $ n\in\N $ satisfying $ n\geq 2 M $. 

In order to obtain \eqref{cond2}, we define the functions
\[
d_n(t):=\frac{nT^{2}}{4}t\;\;\; \text{and} \;\;\; e_n(t):=-\frac{nT^{2}}{4}(t-nT).
\]
We notice that $ d_n(t)>0 $ for $ t>0 $;  $ e_n(t)>0 $ for $ t<nT $; and $ d_n(t)=  e_n(t)$ if, and only if $ t=\tfrac{nT}{2} $. Consider the functions
\[
r^{+}_n(t):=f(t)-d_n(t)\;\;\; \text{and} \;\;\; s^{+}_n(t):= f(t)-e_n(t),
\]
 where $  f(t) $ is the function defined in \eqref{curveshn}. We observe that $r^{+}_n(0)=0 $ and $ s^{+}_n(nT)=0$. In addition, given that $ \alpha(t+\phi_0)=- f'(t) $ for every $ t\in \R $, and taking \eqref{alpha} into account, it follows that 
\[
(r^{+}_n)'(t)<0 \;\;\; \text{and} \;\;\;(s^{+}_n)'(t)>0 \;\;\; \text{for all} \;\;\; t\in \R,
\]
and $ n\geq 2M $. This means that $f(t)<d_n(t) $ for all $ t>0 $, and $ f(t)<e_n(t) $ for all $ t<nT $. Thus, by defining the function 
\[
g^{+}_n(t):=\left\{\begin{aligned}
d_n(t) \;\; &\text{if} \;\; t\in(0,\tfrac{nT}{2}],\\
e_n(t) \;\; &\text{if} \;\; t\in(\tfrac{nT}{2},T),
\end{aligned}\right.
\]
and taking into account that $ g^{+}_n(t)\leq h_n(t) $\footnote{There is a typo in the definition of the function $ g^{+}_n(t) $ in the published paper. Such typo does not compromise the accuracy of the result.} for all $ t\in(0,nT) $, where $ h_n$ is the function defined in \eqref{curveshn}, it follows that $ f(t)< h_n(t) $ for all $ t\in(0,nT) $ and $n\geq 2M$.  In  an analogous way, we can show that $  - h_n(t)<f(t) $ for all $ t\in(0,nT) $ and $ n\geq 2M $. It concludes  this proof.
\end{proof}

\section*{Declarations}

\subsection*{Ethics approval and consent to participate} Not applicable
\subsection*{Consent for publication} Not applicable
\subsection*{Availability of data and materials} Data sharing not applicable to this article as no datasets were generated or analyzed during the current study.
\subsection*{Competing interests} To the best of our knowledge, no conflict of interest, financial or other, exists.
\subsection*{Funding} DDN is partially supported by S\~{a}o Paulo Research Foundation (FAPESP) grants 2022/09633-5,  2019/10269-3, and 2018/13481-0, and by Conselho Nacional de Desenvolvimento Cient\'{i}fico e Tecnol\'{o}gico (CNPq) grant 309110/2021-1. LVMFS is partially supported by S\~{a}o Paulo Research Foundation (FAPESP) grants 2018/22398-0 and 2021/11515-8.
\subsection*{Authors' contributions} All persons who meet authorship criteria are listed as authors, and all authors certify that they have participated sufficiently in the work to take public responsibility for the content, including participation in the conceptualization, methodology, formal analysis, investigation, writing-original draft preparation and writing-review \& editing.
\subsection*{Acknowledgements} Not applicable

\bibliographystyle{abbrv}
\bibliography{references}

\begin{thebibliography}{10}

\bibitem{Anna2017}
A.~M. Barry, E.~WIdiasih, and R.~Mcgehee.
\newblock Nonsmooth frameworks for an extended budyko model.
\newblock {\em Discrete and Continuous Dynamical Systems - B},
  22(6):2447--2463, 2017.

\bibitem{Brogliato1996}
B.~Brogliato.
\newblock {\em Nonsmooth impact mechanics}, volume 220 of {\em Lecture Notes in
  Control and Information Sciences}.
\newblock Springer-Verlag London, Ltd., London, 1996.
\newblock Models, dynamics and control.

\bibitem{Burra2020}
L.~Burra and F.~Zanolin.
\newblock Chaos in a periodically perturbed second-order equation with signum
  nonlinearity.
\newblock {\em Internat. J. Bifur. Chaos Appl. Sci. Engrg.}, 30(2):2050031, 9,
  2020.

\bibitem{Carvalho2020}
T.~Carvalho, D.~D. Novaes, and L.~F. Gon{\c{c}}alves.
\newblock Sliding {S}hilnikov connection in {F}ilippov-type
  predator{\textendash}prey model.
\newblock {\em Nonlinear Dynamics}, 100(3):2973--2987, May 2020.

\bibitem{Zehnder1987}
R.~Dieckerhoff and E.~Zehnder.
\newblock Boundedness of solutions via the twist-theorem.
\newblock {\em Ann. Scuola Norm. Sup. Pisa Cl. Sci. (4)}, 14(1):79--95, 1987.

\bibitem{Ortega2019}
R.~Engui\c{c}a and R.~Ortega.
\newblock Functions with average and bounded motions of a forced discontinuous
  oscillator.
\newblock {\em J. Dynam. Differential Equations}, 31(3):1185--1198, 2019.

\bibitem{Filippov1988}
A.~F. Filippov.
\newblock {\em Differential equations with discontinuous righthand sides},
  volume~18 of {\em Mathematics and its Applications (Soviet Series)}.
\newblock Kluwer Academic Publishers Group, Dordrecht, 1988.
\newblock Translated from the Russian.

\bibitem{Jacquemard2012}
A.~Jacquemard and M.~Teixeira.
\newblock Periodic solutions of a class of non-autonomous second order
  differential equations with discontinuous right-hand side.
\newblock {\em Physica D: Nonlinear Phenomena}, 241(22):2003--2009, 2012.
\newblock Dynamics and Bifurcations of Nonsmooth Systems.

\bibitem{Jeffrey2018}
M.~R. Jeffrey.
\newblock {\em Hidden dynamics}.
\newblock Springer, Cham, 2018.
\newblock The mathematics of switches, decisions and other discontinuous
  behaviour.

\bibitem{Kowalczyk2008}
P.~Kowalczyk and P.~Piiroinen.
\newblock Two-parameter sliding bifurcations of periodic solutions in a
  dry-friction oscillator.
\newblock {\em Physica D: Nonlinear Phenomena}, 237(8):1053--1073, 2008.

\bibitem{Kunze1997}
M.~Kunze, T.~K\"{u}pper, and J.~You.
\newblock On the application of {KAM} theory to discontinuous dynamical
  systems.
\newblock {\em J. Differential Equations}, 139(1):1--21, 1997.

\bibitem{Kurzweil2014}
J.~Kurzweil.
\newblock {\em Ordinary Differential Equations: Introduction to the Theory of
  Ordinary Differential Equations in the Real Domain}.
\newblock ISSN. Elsevier Science, 2014.

\bibitem{Leifeld2018}
J.~Leifeld.
\newblock Non-smooth homoclinic bifurcation in a conceptual climate model.
\newblock {\em European J. Appl. Math.}, 29(5):891--904, 2018.

\bibitem{Li2001}
X.~Li.
\newblock Boundedness of solutions for {D}uffing's equations with semilinear
  potentials.
\newblock {\em J. Differential Equations}, 176(1):248--268, 2001.

\bibitem{Littlewood1966B}
J.~E. Littlewood.
\newblock Unbounded solutions of {$\ddot y+g(y)=p(t)$}.
\newblock {\em J. London Math. Soc.}, 41:491--496, 1966.

\bibitem{Littlewood1968}
J.~E. Littlewood.
\newblock {\em Some problems in real and complex analysis}.
\newblock D. C. Heath and Company Raytheon Education Company, Lexington, Mass.,
  1968.

\bibitem{Morris1976}
G.~R. Morris.
\newblock A case of boundedness in {L}ittlewood's problem on oscillatory
  differential equations.
\newblock {\em Bull. Austral. Math. Soc.}, 14(1):71--93, 1976.

\bibitem{Moser1962}
J.~Moser.
\newblock On invariant curves of area-preserving mappings of an annulus.
\newblock {\em Nachr. Akad. Wiss. G\"{o}ttingen Math.-Phys. Kl. II},
  1962:1--20, 1962.

\bibitem{Ortega1996}
R.~Ortega.
\newblock Asymmetric oscillators and twist mappings.
\newblock {\em J. London Math. Soc. (2)}, 53(2):325--342, 1996.

\bibitem{Ortega1998Talk}
R.~Ortega.
\newblock A report on the boundedness for semilinear {D}uffing’s equation,
  talk.
\newblock Institute of Mathematics, Academia Sinica, 1998.

\bibitem{Piltz2014}
S.~H. Piltz, M.~A. Porter, and P.~K. Maini.
\newblock Prey switching with a linear preference trade-off.
\newblock {\em SIAM Journal on Applied Dynamical Systems}, 13(2):658--682,
  2014.

\bibitem{Sprott2010}
K.~Sun and J.~C. Sprott.
\newblock Periodically forced chaotic system with signum nonlinearity.
\newblock {\em Internat. J. Bifur. Chaos Appl. Sci. Engrg.}, 20(5):1499--1507,
  2010.

\bibitem{Wang2006}
Y.~Wang.
\newblock Boundedness of solutions in a class of {D}uffing equations with a
  bounded restore force.
\newblock {\em Discrete Contin. Dyn. Syst.}, 14(4):783--800, 2006.

\end{thebibliography}

\end{document}